\documentclass[graybox]{svmult}

\usepackage{type1cm}        
\usepackage{makeidx}         
\usepackage{graphicx}        
\usepackage{multicol}        
\usepackage[bottom]{footmisc}
\usepackage{newtxtext}       %
\usepackage[varvw]{newtxmath}       
\makeindex             

\newtheorem{assumption}{Assumption}

\begin{document}

\title*{Coarse spaces using extended generalized eigenproblems for heterogeneous Helmholtz problems}
\titlerunning{Extended harmonic coarse spaces}
\author{
    Emile Parolin\orcidID{0000-0002-6512-7980} and\\
    Frédéric Nataf\orcidID{0000-0003-2813-3481}
}
\institute{Frédéric Nataf \at Sorbonne Université, Université Paris Cité, CNRS, INRIA, Laboratoire
Jacques-Louis Lions, LJLL, EPC ALPINES, F-75005 Paris, France, \email{frederic.nataf@sorbonne-universite.fr}
\and Emile Parolin \at Sorbonne Université, Université Paris Cité, CNRS, INRIA, Laboratoire
Jacques-Louis Lions, LJLL, EPC ALPINES, F-75005 Paris, France \email{emile.parolin@inria.fr}}

\maketitle

\abstract{
	An abstract construction of coarse spaces for non-Hermitian problems and
	non-Hermitian domain decomposition preconditioners based on extended
	generalized eigenproblems was proposed in~\cite{Nataf2025}
	and analyzed on the matrix formulation.
	Building upon this work, we consider instead here the specific case of
	heterogeneous Helmholtz problems, and the derivation and analysis is
	performed at the continuous level.
	Albeit different from its derivation, its use of oversampling and the
	underlying eigenproblems, our approach shares similarities with the methods
	in~\cite{Hu2025,Ma2025b}.
}

\section{Optimized Restricted Additive Schwarz method}

\noindent
\textbf{Model problem.}
Let \(d\in\{1,2,3\}\) and
\(\Omega\subset\mathbb{R}^{d}\) be an open and bounded domain.
We consider the heterogeneous Helmholtz equation for \(\omega>0\)
\begin{equation}\label{eq:model_pb_strong}
    \text{Find}\ u\in H^{1}(\Omega)\ \text{such that:}
    \qquad
    \begin{cases}
        -\nabla\cdot (\mu \nabla u) - \omega^{2} \nu u = f, & \text{in}\ \Omega,\\
        \mu \partial_{\mathbf{n}} u - \imath \omega u = g, & \text{on}\ \partial\Omega,
    \end{cases}
\end{equation}
where \(\mu\) and \(\nu\) are two variable but positive and bounded
coefficients, \(f\in L^{2}(\Omega)\) and \(g\in L^{2}(\partial\Omega)\) model
the source terms, and \(\mathbf{n}\) is the outward unit normal vector to the
boundary \(\partial\Omega\).
The model problem~\eqref{eq:model_pb_strong} takes the following
variational form
\begin{equation}\label{eq:model_pb_weak}
    \text{Find}\ u\in H^{1}(\Omega)\ \text{such that:}
    \qquad
    a_{\Omega}(u,u^{t}) = l_{\Omega}(u^{t}),
    \quad\forall u^{t}\in H^{1}(\Omega),
\end{equation}
where we introduced, for any \(D \subset \Omega\)
and any \(u,u^{t} \in H^{1}(D)\):
\begin{align}
    a_{D}(u,u^{t})
    & := (\mu \nabla u, \nabla u^{t})_{L^{2}(D)^{d}} 
    - \omega^2 (\nu u, u^{t})_{L^{2}(D)}
    - \imath\omega (u, u^{t})_{L^{2}(\partial D\cap\partial\Omega)},\\
    l_{D}(u^{t})
    & := (f, u^{t})_{L^{2}(D)} 
    + (g, u^{t})_{L^{2}(\partial D\cap\partial\Omega)}.
\end{align}
The above model problem~\eqref{eq:model_pb_weak} is well-posed using standard
arguments based on the Fredholm alternative and a unique continuation principle.

\smallskip
\noindent
\textbf{One-level domain decomposition algorithm}.
We introduce a decomposition of \(\Omega\) into \(J\in\mathbb{N}\), \(J>0\), 
overlapping open subdomains \(\Omega_{j}\), \(j=1,\ldots,J\),
associated to a partition-of-unity function
\(\chi_{j} \in W^{1,\infty}(\Omega_{j})\)
such that
\begin{equation}
    \overline{\Omega} = \textstyle\bigcup_{j} \overline{\Omega_{j}},
    \qquad
    \chi_{j} \geq 0,\ 
    \chi_{j}|_{\partial\Omega_{j} \setminus \partial\Omega} = 0,
    \qquad
    \textstyle\sum_{j} E_{j} \chi_{j} u|_{\Omega_{j}} = u, \ \forall u \in H^{1}(\Omega),
\end{equation}
where \(E_{j}\) is an extension-by-continuity operator.
More precisely,
we require that
\(E_{j} : H^{1}(\Omega_{j}) \to H^{1}({\Omega})\) 
continuously,
\(E_{j}u_{j}|_{\Omega_{j}} = u_{j}\)
and
\(\operatorname{supp}(E_{j}u_{j}) \subset \widetilde{\Omega}_{j}\)
for any
\(u_{j} \in H^{1}(\Omega_{j})\)
and for some extended open subdomain
\(\widetilde{\Omega}_{j}\)
such that
\({\Omega}_{j} \subsetneq \widetilde{\Omega}_{j} \subset \Omega\).

Starting from an initial guess
\(u^{0}_{\mathrm{I}} \in H^{1}(\Omega)\),
the one-level (the index \(_{\mathrm{I}}\) stands for one-level) Optimized
Restricted Additive Schwarz (ORAS) algorithm defines the sequence
\begin{equation}\label{eq:oras_fixed_point}
    u^{n+1}_{\mathrm{I}} := u^{n}_{\mathrm{I}} + \textstyle\sum_{j} E_{j} \chi_{j} v_{j}^{n+1},
    \qquad n \geq 0,
\end{equation}
where \(v_{j}^{n+1}\) is solution of:
Find \(v_{j}^{n+1} \in H^{1}(\Omega_{j})\) such that:
\begin{equation}\label{eq:vj}
    b_{\Omega_{j}}(v_{j}^{n+1},v_{j}^{t}) =
    l_{\Omega}(E_{j}v_{j}^{t})
    - a_{\Omega}(u^{n}_{\mathrm{I}}, E_{j}v_{j}^{t}),
    \qquad \forall v^{t}_{j}\in H^{1}(\Omega_{j}),
\end{equation}
where we introduced, for any \(D \subset \Omega\)
and any \(v,v^{t} \in H^{1}(D)\):
\begin{align}
    b_{D}(v,v^{t})
    := 
    a_{D}(v,v^{t})
    - \imath\omega (v, v^{t})_{L^{2}(\partial D\setminus\partial\Omega)}.
\end{align}
By similar arguments as for the model problem, the local
problem~\eqref{eq:vj} is well-posed, and we remark that the right-hand-side
in~\eqref{eq:vj} can be computed locally in \(\widetilde{\Omega}_{j}\) so that
\(v_{j}^{n+1}\) is equivalently solution of
\begin{equation}\label{eq:vj_local}
    b_{\Omega_{j}}(v_{j}^{n+1},v_{j}^{t}) =
    l_{\widetilde{\Omega}_{j}}((E_{j}v_{j}^{t})|_{\widetilde{\Omega}_{j}})
    - a_{\widetilde{\Omega}_{j}}(
        u^{n}_{\mathrm{I}}|_{\widetilde{\Omega}_{j}},
        (E_{j}v_{j}^{t})|_{\widetilde{\Omega}_{j}}
    ),
    \qquad \forall v^{t}_{j}\in H^{1}(\Omega_{j}).
\end{equation}
For general partitions (where at least three subdomains overlap in some part of
the domain \(\Omega\)), the sequence of \(u_{\mathrm{I}}^{n}\) depends a priori on
the choice of \(E_{j}\) (which is not unique), since it affects the
right-hand-side of~\eqref{eq:vj} (in~\eqref{eq:oras_fixed_point}, the presence
of the partition-of-unity \(\chi_{j}\) implies that \(E_{j}\) extends by zero,
which is uniquely defined).

\smallskip
\noindent
\textbf{One-level error analysis.}
For \(n \geq 0\), the one-level error
\(e^{n}_{\mathrm{I}}:= u - u^{n}_{\mathrm{I}}\) satisfies
\begin{equation}
    e^{n+1}_{\mathrm{I}} = u - u^{n+1}_{\mathrm{I}}
    = e^{n}_{\mathrm{I}} - \textstyle\sum_{j} E_{j} \chi_{j} v_{j}^{n+1}
    = \textstyle\sum_{j} E_{j} \chi_{j} (e^{n}_{\mathrm{I}}|_{\Omega_{j}} - v_{j}^{n+1}).
\end{equation}
From~\eqref{eq:vj_local}, we have,
\begin{equation}
    b_{\Omega_{j}}(e^{n}_{\mathrm{I}}|_{\Omega_{j}} - v_{j}^{n+1},v_{j}^{t}) =
    b_{\Omega_{j}}(e^{n}_{\mathrm{I}}|_{\Omega_{j}},v_{j}^{t}) -
    a_{\widetilde{\Omega}_{j}}(
        e^{n}_{\mathrm{I}}|_{\widetilde{\Omega}_{j}},
        (E_{j}v_{j}^{t})|_{\widetilde{\Omega}_{j}}
    ),
    \ \forall v_{j}^{t} \in H^{1}(\Omega_{j}),
\end{equation}
so that
\begin{equation}
    e^{n+1}_{\mathrm{I}}
    = \textstyle\sum_{j} E_{j}\chi_{j}R_{j}e^{n}_{\mathrm{I}}|_{\widetilde{\Omega}_{j}},
\end{equation}
where \(R_{j} : H^{1}(\widetilde{\Omega}_{j}) \to H^{1}(\Omega_{j})\) is
defined for any \(\widetilde{v}_{j} \in H^{1}(\widetilde{\Omega}_{j})\) as
\begin{equation}\label{eq:def_Rj}
    b_{\Omega_{j}}(R_{j} \widetilde{v}_{j}, v_{j}^{t}) =
    b_{\Omega_{j}}(\widetilde{v}_{j}|_{\Omega_{j}}, v_{j}^{t}) -
    a_{\widetilde{\Omega}_{j}}(\widetilde{v}_{j}, (E_{j}v_{j}^{t})|_{\widetilde{\Omega}_{j}}),
    \qquad \forall v_{j}^{t} \in H^{1}(\Omega_{j}).
\end{equation}
The operator \(R_{j}\) is well-defined since the local
problems~\eqref{eq:def_Rj} are well-posed.

We measure the error in the energy norm given 
by the following sesquilinear form:
for any \(D \subset \Omega\)
and any \(u,u^{t} \in H^{1}(D)\), let
\begin{equation}\label{eq:norm}
    c_{D}(u,u^{t})
    := (\mu \nabla u, \nabla u^{t})_{L^{2}(D)^{d}} 
    + \omega^2 (\nu u, u^{t})_{L^{2}(D)},
    \qquad 
    \|u\|_{D}^{2} := c_{D}(u,u).
\end{equation}
The error can be estimated as follows
\begin{align}\label{eq:error_est_lvl1}
    \|e^{n+1}_{\mathrm{I}}\|_{\Omega}^{2}
     \leq k_{0}
    \textstyle\sum_{j} 
    \big\|\chi_{j} R_{j}e^{n}_{\mathrm{I}}|_{\widetilde{\Omega}_{j}}\big\| _{\Omega_{j}}^{2}
     \leq k_{0} \xi
    \textstyle\sum_{j} 
    \big\|e^{n}_{\mathrm{I}}|_{\widetilde{\Omega}_{j}}\big\|_{\widetilde{\Omega}_{j}}^{2}
    \leq k_{0} k_{1} \xi
    \|e^{n}_{\mathrm{I}}\|_{\Omega}^{2},
\end{align}
where we introduced
\begin{align}
    & k_{0} := \!\!\sup_{\substack{(u_{j})_{j}\\0 \neq u_{j} \in H^{1}_{0}(\Omega_{j})}}\!\!
    \frac{
        \|\textstyle\sum_{j} E_{j}u_{j}\|_{\Omega}^{2}
    }{
        \textstyle\sum_{j} \|u_{j}\|_{\Omega_{j}}^{2}
    },
    \qquad
    k_{1} := \!\!\sup_{0 \neq u \in H^{1}(\Omega)}\!\!
    \frac{
        \textstyle\sum_{j} \|u|_{\widetilde{\Omega}_{j}}\|_{\widetilde{\Omega}_{j}}^{2}
    }{
        \|u\|_{\Omega}^{2}
    },\label{eq:k01}\\
    & \xi := \sup_{j} \xi_{j},
    \qquad
    \xi_{j} := \!\!\sup_{0 \neq \widetilde{u}_{j} \in H^{1}(\widetilde{\Omega}_{j})}\!\!
    \frac{
        \|\chi_{j} R_{j}\widetilde{u}_{j}\|_{\Omega_{j}}^{2}
    }
    {
        \|\widetilde{u}_{j}\|_{\widetilde{\Omega}_{j}}^{2}
    }.
\end{align}
The estimate~\eqref{eq:error_est_lvl1} shows that a sufficient condition to
have convergence of the algorithm is that \(k_{0}k_{1}\xi < 1\), which is
unlikely to be satisfied.
The two constants \(k_{0}\) and \(k_{1}\) can be bounded uniformly of the
parameters \(\mu,\nu,\omega\) of the problem and only depend respectively on the
maximum number of neighbors for a subdomain
(see e.g.~\cite[Lem.~7.9]{Dolean2015}) 
and the maximum number of subdomain intersection
multiplicity
(see e.g.~\cite[Lem.~7.13]{Dolean2015}).
For reasonable partitions, both quantities can be bounded uniformly of the
number of subdomains.
On the other hand, \(\xi\) depends on the parameters of the problem.
In the next section, we propose the construction of a coarse space and
associated two-level domain decomposition algorithm that allows to control each
\(\xi_{j}\) independently, hence \(\xi\).

\section{Two-level ORAS method}

\noindent
\textbf{Local generalized eigenproblem}
The definition of \(\xi_{j}\) suggests considering the following local
generalized eigenproblem:
Find
\(
(\widetilde{u}_{j},\lambda_{j}) \in H^{1}(\widetilde{\Omega}_{j}) \times \mathbb{R}
\)
such that:
\begin{align}\label{eq:gevp}
    c_{\Omega_{j}}(\chi_{j} R_{j}\widetilde{u}_{j}, \chi_{j} R_{j}\widetilde{u}_{j}^{t})
    = \lambda_{j}
    c_{\widetilde{\Omega}_{j}}(\widetilde{u}_{j}, \widetilde{u}_{j}^{t}),
    \qquad\forall
    \widetilde{u}_{j}^{t} \in H^{1}(\widetilde{\Omega}_{j}).
\end{align}
This problem is equivalent to the following saddle-point formulation:\\
Find
\(
(\widetilde{u}_{j},v_{j},\sigma_{j},\lambda_{j})
\in H^{1}(\widetilde{\Omega}_{j}) \times H^{1}(\Omega_{j}) \times H^{1}(\Omega_{j}) \times \mathbb{R}
\)
such that:
\begin{align}\label{eq:gevp_saddlepoint}
    \begin{cases}
        c_{\Omega_{j}}(\chi_{j} v_{j}, \chi_{j} v_{j}^{t})
        - b_{\Omega_{j}}(v_{j}^{t}, \sigma_{j})
        = 0,
        & \forall
        v_{j}^{t} \in H^{1}(\Omega_{j}),\\
        - b_{\Omega_{j}}(v_{j}, \sigma_{j}^{t})
        + b_{\Omega_{j}}(\widetilde{u}_{j}|_{\Omega_{j}}, \sigma_{j}^{t})
        - a_{\widetilde{\Omega}_{j}}(\widetilde{u}_{j}, (E_{j}\sigma_{j}^{t})|_{\widetilde{\Omega}_{j}})
        = 0,
        & \forall
        \sigma_{j}^{t} \in H^{1}(\Omega_{j}),\\
        b_{\Omega_{j}}(\widetilde{u}_{j}^{t}|_{\Omega_{j}}, \sigma_{j})
        - a_{\widetilde{\Omega}_{j}}(\widetilde{u}_{j}^{t}, (E_{j}\sigma_{j})|_{\widetilde{\Omega}_{j}})
        = \lambda_{j}
        c_{\widetilde{\Omega}_{j}}(\widetilde{u}_{j}, \widetilde{u}_{j}^{t}),
        & \forall
        \widetilde{u}_{j}^{t} \in H^{1}(\widetilde{\Omega}_{j}).
    \end{cases}
\end{align}

The following result shows that each \(\xi_{j}\) (hence \(\xi\)) is finite and
that the eigenvalues \(\lambda_{j}\) solutions to~\eqref{eq:gevp} 
converge to \(0\).

\begin{proposition}\label{th:Pj_compact}
    The operator
    \(
        K_{j} : \widetilde{u}_{j} \in H^{1}(\widetilde{\Omega}_{j}) \mapsto
        \chi_{j} R_{j} \widetilde{u}_{j} \in H^{1}(\Omega_{j})
    \)
    is compact.
\end{proposition}
\begin{proof}
    The result is based on a Caccioppoli inequality which we establish next
    (see also~\cite[Lem.~2.7]{Ma2023}).
    From~\eqref{eq:def_Rj} we have, 
    for any \(\widetilde{u}_{j} \in H^{1}(\widetilde{\Omega}_{j})\),
    \begin{equation}\label{eq:start_Caccioppoli}
        b_{\Omega_{j}}(R_{j} \widetilde{u}_{j}, \chi_{j}^{2} R_{j} \widetilde{u}_{j}) = 0,
    \end{equation}
    which implies (adding~\eqref{eq:start_Caccioppoli} and its conjugate) that
    \begin{align}
        & (\mu \chi_{j} \nabla R_{j} \widetilde{u}_{j}, \chi_{j} \nabla R_{j} \widetilde{u}_{j})_{L^{2}(\Omega_{j})^{d}} 
        + (\mu R_{j} \widetilde{u}_{j} \nabla \chi_{j}, \chi_{j} \nabla R_{j} \widetilde{u}_{j})_{L^{2}(\Omega_{j})^{d}} \nonumber\\
        &\quad + (\mu \chi_{j} \nabla R_{j} \widetilde{u}_{j}, R_{j} \widetilde{u}_{j} \nabla \chi_{j})_{L^{2}(\Omega_{j})^{d}} 
        - \omega^2 (\nu R_{j} \widetilde{u}_{j}, \chi_{j}^{2} R_{j} \widetilde{u}_{j})_{L^{2}(\Omega_{j})} = 0.
    \end{align}
    On the other hand, we have, for any
    \(
    \widetilde{u}_{j} \in H^{1}(\widetilde{\Omega}_{j})
    \),
    \begin{align}
        & \|\chi_{j} R_{j}\widetilde{u}_{j}\|_{\Omega_{j}}^{2} = 
        c_{\Omega_{j}}(\chi_{j} R_{j}\widetilde{u}_{j}, \chi_{j} R_{j}\widetilde{u}_{j}) = 
         (\mu R_{j} \widetilde{u}_{j} \nabla \chi_{j}, R_{j} \widetilde{u}_{j} \nabla \chi_{j})_{L^{2}(\Omega_{j})^{d}} \nonumber\\
        &\quad + (\mu \chi_{j} \nabla R_{j} \widetilde{u}_{j}, \chi_{j} \nabla R_{j} \widetilde{u}_{j})_{L^{2}(\Omega_{j})^{d}} 
        + (\mu R_{j} \widetilde{u}_{j} \nabla \chi_{j}, \chi_{j} \nabla R_{j} \widetilde{u}_{j})_{L^{2}(\Omega_{j})^{d}} \nonumber\\
        &\quad + (\mu \chi_{j} \nabla R_{j} \widetilde{u}_{j}, R_{j} \widetilde{u}_{j} \nabla \chi_{j})_{L^{2}(\Omega_{j})^{d}} 
        + \omega^2 (\nu \chi_{j} R_{j} \widetilde{u}_{j}, \chi_{j} R_{j} \widetilde{u}_{j})_{L^{2}(\Omega_{j})}\nonumber\\
        & = (\mu R_{j} \widetilde{u}_{j} \nabla \chi_{j}, R_{j} \widetilde{u}_{j} \nabla \chi_{j})_{L^{2}(\Omega_{j})^{d}} 
        + 2\omega^2 (\nu R_{j} \widetilde{u}_{j}, \chi_{j}^{2} R_{j} \widetilde{u}_{j})_{L^{2}(\Omega_{j})}
        \leq C \|R_{j} \widetilde{u}_{j}\|^{2}_{L^{2}(\Omega_{j})},\nonumber
    \end{align}
    for some constant \(C>0\) depending on \(\omega\), upper bounds on the
    coefficients \(\mu,\nu\) and the partition-of-unity \(\chi_{j}\).
    The result then follows by application of the Rellich--Kondrachov theorem.
\end{proof}

\noindent
\textbf{Coarse space.}
For a user-defined parameter \(\tau > 0\),
we introduce the coarse space
\begin{equation}
    U_{\tau} := \textstyle\bigoplus_{j} \operatorname{span}\{
        E_{j} \chi_{j} R_{j} \widetilde{u}_{j}\ :
        \ (\widetilde{u}_{j},\lambda_{j})\ 
        \text{solution of~\eqref{eq:gevp} such that}\ \lambda_{j} > \tau
    \}.
\end{equation}
The coarse space \(U_{\tau}\) is made of local Helmholtz solutions
\(R_{j}\widetilde{u}_{j}\) that are extended-by-zero globally thanks to
\(E_{j}\) after being multiplied by the partition-of-unity function
\(\chi_{j}\).
Indeed, for \(\widetilde{u}_{j}\) solution of~\eqref{eq:gevp},
we have \(R_{j}\widetilde{u}_{j} = v_{j}\) solution
of~\eqref{eq:gevp_saddlepoint} which satisfies
(testing with \(\sigma_{j}^{t} \in H^{1}_{0}(\Omega_{j})\) in~\eqref{eq:gevp_saddlepoint})
\begin{equation}
    (\mu \nabla v_{j}, \nabla v_{j}^{t})_{L^{2}(\Omega_{j})^{d}}
    - \omega^2 (\nu v_{j}, v_{j}^{t})_{L^{2}(\Omega_{j})} = 0,
    \qquad\forall v_{j}^{t} \in H^{1}_{0}(\Omega_{j}).
\end{equation}
Moreover, thanks to the partition-of-unity, the coarse space is conforming,
namely \(U_{\tau} \subset H^{1}(\Omega)\).
Finally, Proposition~\ref{th:Pj_compact} ensures that \(U_{\tau}\)
is finite dimensional.

\smallskip
\noindent
\textbf{Two-level domain decomposition algorithm.}
Starting from an initial guess \(u^{0}_{\mathrm{II}} \in H^{1}(\Omega)\), we define
a sequence of solutions \(u^{n}_{\mathrm{II}} \in H^{1}(\Omega)\)
(the index \(_{\mathrm{II}}\) stands for two-level),
\begin{equation}\label{eq:oras_2lvl_fixedpoint}
    u^{n+1}_{\mathrm{II}} := u^{n}_{\mathrm{II}} + \textstyle\sum_{j} E_{j} \chi_{j} w_{j}^{n+1} + w_{0}^{n+1},
    \qquad n \geq 0,
\end{equation}
where \(w_{j}^{n+1}\) is solution of:
Find \(w_{j}^{n+1} \in H^{1}(\Omega_{j})\) such that:
\begin{equation}\label{eq:wj}
    b_{\Omega_{j}}(w_{j}^{n+1},w_{j}^{t}) =
    l_{\Omega}(E_{j}w_{j}^{t})
    - a_{\Omega}(u^{n}_{\mathrm{II}}, E_{j}w_{j}^{t}),
    \qquad \forall w^{t}_{j}\in H^{1}(\Omega_{j}),
\end{equation}
and \(w_{0}^{n+1}\) is solution of:
Find \(w_{0}^{n+1} \in U_{\tau}\) such that:
\begin{equation}\label{eq:w0}
    a_{\Omega}(w_{0}^{n+1}, w_{0}^{t}) =
    l_{\Omega}(w_{0}^{t})
    - a_{\Omega}(u^{n}_{\mathrm{II}} + \textstyle\sum_{j} E_{j} \chi_{j} w_{j}^{n+1},w_{0}^{t}),
    \quad\forall w_{0}^{t} \in U_{\tau}.
\end{equation}
\begin{assumption}\label{hyp:coarse_pb_wellposed}
    For any anti-linear form \(L\) on \(H^{1}(\Omega)\), consider the two problems
    \begin{alignat}{3}
        & \text{Find } u\in H^{1}(\Omega) \text{ such that: }\quad
        && a_{\Omega}(u,u^t) = L(u^{t}), \ \forall u^t \in H^{1}(\Omega),\\
        & \text{Find } u_{0}\in U_{\tau} \text{ such that: }\quad
        && a_{\Omega}(u_{0},u_{0}^t) = L(u_{0}^{t}), \ \forall u_{0}^t \in U_{\tau}.\label{eq:general_coarse_pb}
    \end{alignat}
    In the following, we assume that the coarse
    problem~\eqref{eq:general_coarse_pb} is well-posed and that there exists
    \(\sigma>0\), independent of \(\tau\) and \(L\),
    such that the following stability bound holds
    \begin{equation}\label{eq:coarse_pb_stability}
        \|u-u_{0}\|_{\Omega} \leq \sigma \|u\|_{\Omega}.
    \end{equation}
\end{assumption}
Assumption~\ref{hyp:coarse_pb_wellposed} implies that the problem~\eqref{eq:w0}
is well-posed.
The validity of Assumption~\ref{hyp:coarse_pb_wellposed} can be argued
for large enough coarse spaces \(U_{\tau}\) (i.e.\ taking the
user-controlled parameter \(\tau\) sufficiently small).
For instance, well-posedness of the coarse problem and a quasi-optimality
estimate implying~\eqref{eq:coarse_pb_stability} can be obtained using the
so-called Aubin--Nitsche duality trick and Schatz argument,
see~\cite[Lemma~5.1]{Galkowski2025}.

\smallskip
\noindent
\textbf{Two-level error analysis.}
For \(n \geq 0\), the two-level error
\(e^{n}_{\mathrm{II}}:= u - u^{n}_{\mathrm{II}}\) satisfies
\begin{equation}\label{eq:error_oras_2lvl}
    e^{n+1}_{\mathrm{II}} = u - u^{n+1}_{\mathrm{II}}
    = u - u^{n}_{\mathrm{II}} - \textstyle\sum_{j} E_{j} \chi_{j} w_{j}^{n+1} - w_{0}^{n+1}.
\end{equation}
From~\eqref{eq:wj}, we have,
\begin{equation}
    b_{\Omega_{j}}(e^{n}_{\mathrm{II}}|_{\Omega_{j}} - w_{j}^{n+1},w_{j}^{t}) =
    b_{\Omega_{j}}(e^{n}_{\mathrm{II}}|_{\Omega_{j}},w_{j}^{t}) -
    a_{\widetilde{\Omega}_{j}}(e^{n}_{\mathrm{II}}|_{\widetilde{\Omega}_{j}}, (E_{j}w_{j}^{t})|_{\widetilde{\Omega}_{j}}),
    \ \forall w_{j}^{t} \in H^{1}(\Omega_{j}),
\end{equation}
and from~\eqref{eq:w0}, we have
\begin{equation}
    a_{\Omega}(w_{0}^{n+1}, w_{0}^{t}) =
    a_{\Omega}(\textstyle\sum_{j} E_{j} \chi_{j} (e^{n}_{\mathrm{II}}|_{\Omega_{j}} - w_{j}^{n+1}),w_{0}^{t}),
    \quad\forall w_{0}^{t} \in U_{\tau},
\end{equation}
so that
\begin{equation}
    e^{n+1}_{\mathrm{II}}
    = (\mathrm{I}-P_{0})\textstyle\sum_{j} E_{j}\chi_{j}R_{j}e^{n}_{\mathrm{II}}|_{\widetilde{\Omega}_{j}},
\end{equation}
where we introduced \(P_{0} : H^{1}(\Omega) \to H^{1}(\Omega)\)
defined as:
for any \(u \in H^{1}(\Omega)\), \(P_{0}u \in U_{\tau}\) such that:
\begin{equation}\label{eq:P0}
    a_{\Omega}(P_{0}u, u_{0}^{t}) = a_{\Omega}(u, u_{0}^{t}),
    \quad\forall u_{0}^{t} \in U_{\tau}.
\end{equation}
Provided Assumption~\ref{hyp:coarse_pb_wellposed} holds,
the operator \(P_{0}\) is well-defined.
Besides, \(P_{0}\) is an oblique projection on the closed space \(U_{\tau}\)
and the operator norm of \(\mathrm{I}-P_{0}\) for \(c_{\Omega}\) is bounded by
the stability constant \(\sigma\) (see~\eqref{eq:coarse_pb_stability}) of the
coarse problem, namely
\begin{equation}\label{eq:ImP0norm}
    \sup_{0 \neq u \in H^{1}(\Omega)}\!\!
    \frac{
        \|(\mathrm{I}-P_{0})u\|_{\Omega}
    }
    {
        \|u\|_{\Omega}
    } \leq \sigma.
\end{equation}

We can now state our main result, which is the counterpart for the
heterogeneous Helmholtz problem of the matrix formulation result stated
in~\cite[Th.~2.12]{Nataf2025}.
\begin{theorem}\label{th:convergence}
    If the user-controlled parameter \(\tau\) is taken to be sufficiently small
    such that Assumption~\ref{hyp:coarse_pb_wellposed} holds
    and \(\rho := \sigma \sqrt{k_{0}k_{1}\tau} < 1\),
    then the iterative solution \(u^{n}_{\mathrm{II}}\) defined in~\eqref{eq:oras_2lvl_fixedpoint}
    converges towards the exact solution \(u\) of the model
    problem~\eqref{eq:model_pb_weak}, in the norm given by \(c_{\Omega}\)
    defined in~\eqref{eq:norm}, and with rate \(\rho\),
    namely
    \begin{equation}
        \|u - u^{n}_{\mathrm{II}}\|_{\Omega}^{} \leq \rho^{n}
        \|u - u^{0}_{\mathrm{II}}\|_{\Omega}^{},
        \qquad n \geq 0.
    \end{equation}
\end{theorem}
\begin{proof}
    Let \(\Pi_{j}\) be the \(c_{\widetilde{\Omega}_{j}}\)-orthogonal projection on
    the finite dimensional space
    \begin{equation}
        \operatorname{span}\{
                \widetilde{u}_{j} \ :
                \ (\widetilde{u}_{j},\lambda_{j})\ 
                \text{solution of~\eqref{eq:gevp} such that}\ \lambda_{j} > \tau
        \}.
    \end{equation}
    Then, by construction of the coarse space,
    the image of
    \(
        \textstyle\sum_{j} E_{j}\chi_{j}R_{j}\Pi_{j}
    \)
    is in the kernel of \(\mathrm{I}-P_{0}\),
    hence
    \begin{equation}
        e^{n+1}_{\mathrm{II}}
        = (\mathrm{I}-P_{0})\textstyle\sum_{j} E_{j}\chi_{j}R_{j}[\Pi_{j} + (\mathrm{I} - \Pi_{j})]e^{n}_{\mathrm{II}}|_{\widetilde{\Omega}_{j}}
        = (\mathrm{I}-P_{0})\textstyle\sum_{j} E_{j}\chi_{j}R_{j}(\mathrm{I} - \Pi_{j})e^{n}_{\mathrm{II}}|_{\widetilde{\Omega}_{j}},
    \end{equation}
    and moreover (see e.g.~\cite[Lem.~7.7]{Dolean2015})
    \begin{equation}
        \|\chi_{j}R_{j} (\mathrm{I} - \Pi_{j})e^{n}_{\mathrm{II}}|_{\widetilde{\Omega}_{j}}\|_{\Omega_{j}}^{2}
        \leq \tau
        \|(\mathrm{I} - \Pi_{j})e^{n}_{\mathrm{II}}|_{\widetilde{\Omega}_{j}}\|_{\widetilde{\Omega}_{j}}^{2}
        \leq \tau
        \|e^{n}_{\mathrm{II}}|_{\widetilde{\Omega}_{j}}\|_{\widetilde{\Omega}_{j}}^{2}.
    \end{equation}
    Recalling~\eqref{eq:k01} and~\eqref{eq:ImP0norm},
    the error can hence be estimated as follows
    \begin{align}\label{eq:error_est_lvl2}
        \|e^{n+1}_{\mathrm{II}}\|_{\Omega}^{2}
        & = \big\|(\mathrm{I}-P_{0})\textstyle\sum_{j} E_{j}\chi_{j}R_{j}(\mathrm{I} - \Pi_{j})e^{n}_{\mathrm{II}}|_{\widetilde{\Omega}_{j}}\big\|_{\Omega}^{2} \nonumber\\
        & \leq \sigma^{2} \;
        \big\|\textstyle\sum_{j} E_{j}\chi_{j}R_{j}(\mathrm{I} - \Pi_{j})e^{n}_{\mathrm{II}}|_{\widetilde{\Omega}_{j}}\big\|_{\Omega}^{2}
        \leq \sigma^{2} k_{0} \;\
        \textstyle\sum_{j} \big\|\chi_{j}R_{j} (\mathrm{I} - \Pi_{j})e^{n}_{\mathrm{II}}|_{\widetilde{\Omega}_{j}}\big\|_{\Omega_{j}}^{2}
        \nonumber\\
        & \leq \sigma^{2} k_{0} \tau \;
        \textstyle\sum_{j} \big\|e^{n}_{\mathrm{II}}|_{\widetilde{\Omega}_{j}}\big\|_{\widetilde{\Omega}_{j}}^{2}
        \leq \sigma^{2} k_{0} k_{1} \tau \;
        \|e^{n}_{\mathrm{II}}\|_{\Omega}^{2}.
    \end{align}
\end{proof}

In practice, Krylov acceleration is used, and one iteration of the above
algorithm~\eqref{eq:oras_2lvl_fixedpoint} is used as a preconditioner.
Theorem~\ref{th:convergence} also implies convergence of the Krylov iterative
method with this preconditioner.

\section{Numerical illustrations}

We conclude with some numerical experiments on the Marmousi model,
which represents a realistic geological structure with varying velocity \(c\)
(see Figure~\ref{fig:marmousi_plots})
and fits within the model problem~\eqref{eq:model_pb_strong} with \(\nu=c^{-2}\),
\(\mu=1\).
A homogeneous Dirichlet boundary condition is applied on the surface
and Robin boundary conditions are imposed on the remaining boundaries.
The source term \(f\) is a regularized Dirac mass (sharp Gaussian function)
centered near the surface (and \(g=0\)).

The implementation is performed in FreeFEM using the FFDDM framework.
The discretization consists of \(\mathbb{P}_{2}\) Lagrange finite elements.
We perform a weak scaling test using a frequency ramp, so that the number of
degrees of freedom in each subdomain remains constant around \(\sim36000\) (\(\sim780\)
on the subdomain boundary).
We compare results for the one-level ORAS preconditioner and the two-level
preconditioners for three coarse space sizes.
In practice, instead of using the parameter \(\tau\) to filter the eigenvalues,
we compute a fixed number of eigenpairs per subdomain given as a percentage of
the number of degrees of freedom on the boundary of the subdomain (which is an
upper bound on the meaningful eigenvectors to include on the coarse space).
This strategy avoids computing too many unnecessary eigenpairs and allows more
sensible run time comparisons.
The numerical results are reported in Table~\ref{tab:marmousi} and in
Figure~\ref{fig:marmousi_scalability}.

The results show the poor scalability of the one-level preconditioner for
increasing \(\omega\).
For a large enough coarse space, the second-level preconditioner converges 
in a much reduced number of iterations, albeit still growing with respect to
\(\omega\).
This is however achieved at the expense of a larger setup time (including
factorization of local matrices, computation of eigenvalues, assembly and
factorization of the coarse problem).
At some point, the reduction in iterations obtained by increasing further the
coarse space is offset by the larger computational effort for building the
coarse space and solving the coarse problem.

Additional numerical results providing a numerical comparison with alternative
methods (including the coarse spaces in~\cite{Hu2025,Ma2025b}) are given
in~\cite{Dolean2025}.

\begin{table}
	\scriptsize
    \centering
    \begin{tabular}{ ccc|ccc|cccc|cccc|cccc }
    \multicolumn{3}{c|}{} & \multicolumn{3}{c|}{one-level} & \multicolumn{4}{c|}{two-level $(7.5\%)$} & \multicolumn{4}{c|}{two-level $(10\%)$} & \multicolumn{4}{c}{two-level $(12.5\%)$} \\
     $\omega$ & $J$ & $N$ & It & T$_{\mathrm{setup}}$ & T$_{\mathrm{solver}}$ & CS & It & T$_{\mathrm{setup}}$ & T$_{\mathrm{solver}}$ & CS & It & T$_{\mathrm{setup}}$ & T$_{\mathrm{solver}}$ & CS & It & T$_{\mathrm{setup}}$ & T$_{\mathrm{solver}}$\\
    \hline
     30$\pi$ & 32 & 1106441 & 224 & 1.6 & 37  & 1884 & 25 & 55 & 5.4  & 2512 & 16 & 79 & 4.1  & 3140 & 13 & 105 & 3.9 \\
     42$\pi$ & 64 & 2167257 & 314 & 1.7 & 62  & 3732 & 46 & 55 & 9.3  & 4976 & 20 & 80 & 5.0  & 6220 & 15 & 107 & 4.2 \\
     60$\pi$ & 128 & 4420881 & 478 & 1.5 & 125  & 7524 & 94 & 58 & 20  & 10032 & 27 & 84 & 6.7  & 12540 & 20 & 113 & 5.7 \\
     84$\pi$ & 256 & 8662193 & 645 & 1.7 & 267  & 14865 & 202 & 67 & 70  & 19820 & 42 & 96 & 17  & 24775 & 26 & 134 & 14 \\
     120$\pi$ & 512 & 17673761 & 915 & 2.2 & 1120  & 30297 & 389 & 74 & 178  & 40396 & 62 & 109 & 28  & 50495 & 37 & 150 & 21 
    \end{tabular}
    \caption{
        Angular frequency \(\omega\), number of subdomains \(J\), number of degrees of
        freedom \(N\), number of GMRES iterations to reach a relative residual
        of \(10^{-8}\) (It), setup time (T\(_{\mathrm{setup}}\)) and solver
        time (T\(_{\mathrm{solver}}\)) for the
        one-level preconditioner and two-level
        preconditioners for three coarse space sizes (CS) given
        as a percentage of the number of degrees of freedom on the skeleton
        (X\%).
    }\label{tab:marmousi}
\end{table}

\begin{figure}
    \centering
    {\includegraphics[width=0.85\textwidth]{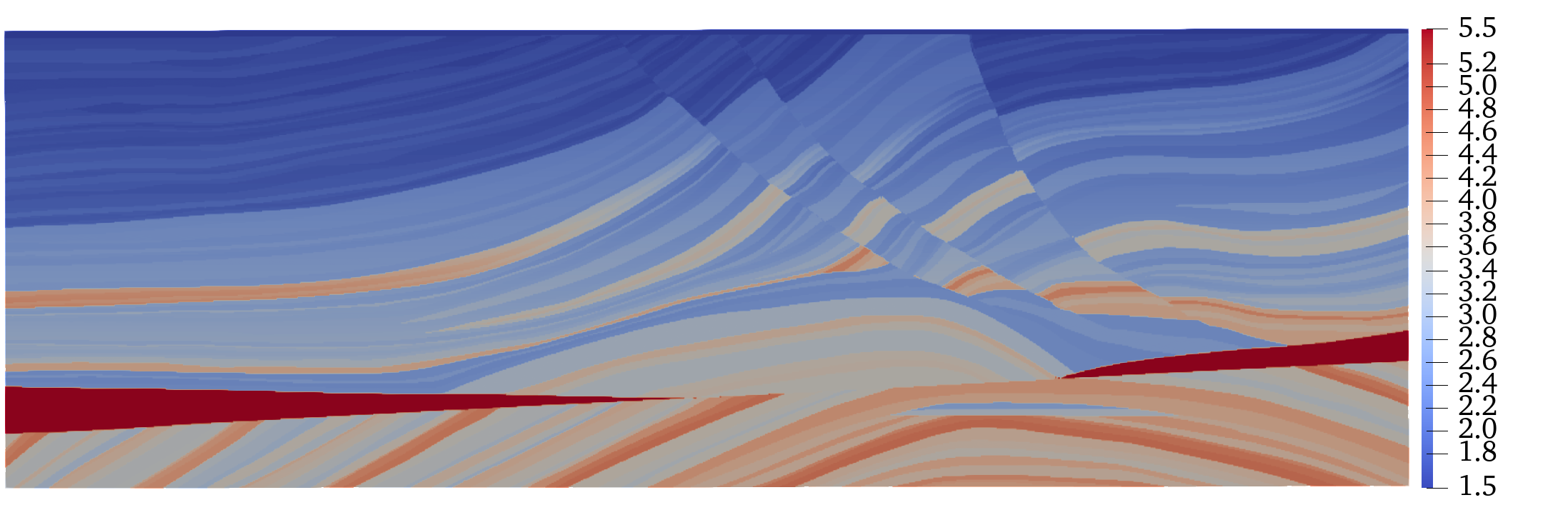}}
    {\includegraphics[width=0.85\textwidth]{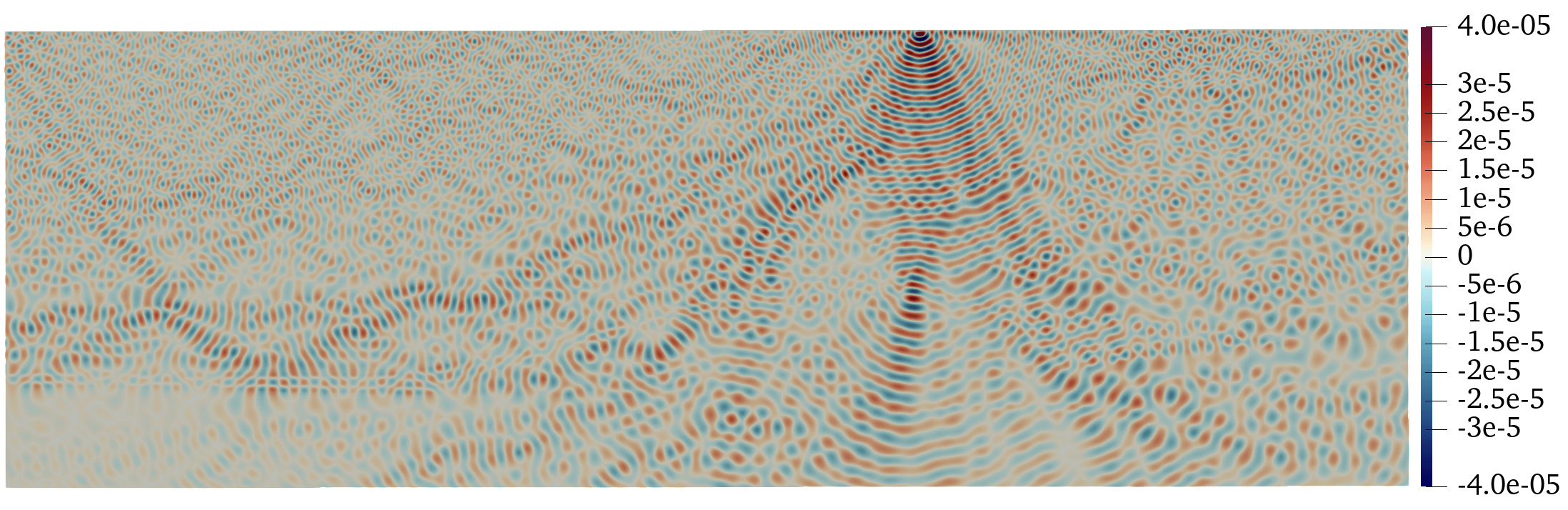}}
    \caption{
        Velocity profile (top) and real part of the solution \(u\) for
        \(\omega=60\pi\) (bottom).
    }\label{fig:marmousi_plots}
\end{figure}

\begin{figure}
    \centering
    {\includegraphics[height=0.16\textheight]{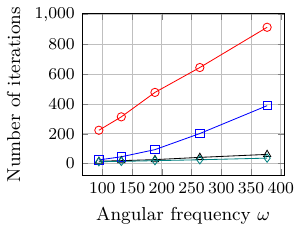}}
    {\includegraphics[height=0.16\textheight]{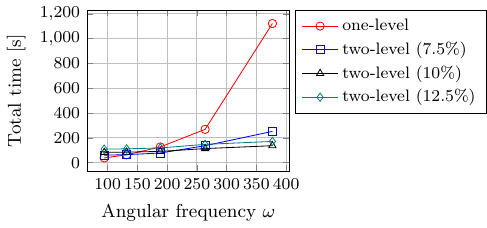}}
    \caption{
        GMRES iteration count (left) and total run time (right) with respect to \(\omega\).
    }\label{fig:marmousi_scalability}
\end{figure}

\bibliographystyle{spmpsci}
\bibliography{biblio.bib}

@book{Dolean2015,
    AUTHOR = {Dolean, Victorita and Jolivet, Pierre and Nataf,
              Fr\'{e}d\'{e}ric},
     TITLE = {An introduction to domain decomposition methods},
      NOTE = {Algorithms, theory, and parallel implementation},
 PUBLISHER = {Society for Industrial and Applied Mathematics (SIAM),
              Philadelphia, PA},
      YEAR = {2015},
     PAGES = {x+238},
      ISBN = {978-1-611974-05-8},
   MRCLASS = {65-02 (65M55 65N55)},
  MRNUMBER = {3450068},
MRREVIEWER = {Benjamin\ Wi-Lian\ Ong},
       DOI = {10.1137/1.9781611974065.ch1},
}

@article{Dolean2025,
      title={Achieving wavenumber robustness in domain decomposition for heterogeneous Helmholtz equation: an overview of spectral coarse spaces}, 
      author={Victorita Dolean and Mark Fry and Matthias Langer and Emile Parolin and Pierre-Henri Tournier},
      journal={arXiv preprint arXiv:2509.02131},
      year={2025},
}

@article{Galkowski2025,
      title={Convergence theory for two-level hybrid Schwarz preconditioners for high-frequency Helmholtz problems}, 
      author={Jeffrey Galkowski and Euan A. Spence},
      journal={arXiv preprint arXiv:2501.11060},
      year={2025},
}

@article{Hu2025,
    AUTHOR = {Hu, Qiya and Li, Ziyi},
     TITLE = {A hybrid two-level weighted {S}chwarz method for {H}elmholtz
              equations},
   JOURNAL = {SIAM J. Numer. Anal.},
  FJOURNAL = {SIAM Journal on Numerical Analysis},
    VOLUME = {63},
      YEAR = {2025},
    NUMBER = {2},
     PAGES = {716--743},
      ISSN = {0036-1429,1095-7170},
   MRCLASS = {65N55 (65N30)},
  MRNUMBER = {4890903},
       DOI = {10.1137/24M1637994},
}

@article{Ma2023,
    AUTHOR = {Ma, Chupeng and Alber, Christian and Scheichl, Robert},
     TITLE = {Wavenumber explicit convergence of a multiscale generalized
              finite element method for heterogeneous {H}elmholtz problems},
   JOURNAL = {SIAM J. Numer. Anal.},
  FJOURNAL = {SIAM Journal on Numerical Analysis},
    VOLUME = {61},
      YEAR = {2023},
    NUMBER = {3},
     PAGES = {1546--1584},
      ISSN = {0036-1429,1095-7170},
   MRCLASS = {65M60 (65N15 65N55)},
  MRNUMBER = {4601686},
       DOI = {10.1137/21M1466748},
}

@article{Ma2025b,
    AUTHOR = {Ma, Chupeng and Alber, Christian and Scheichl, Robert and
              Zhang, Yongwei},
     TITLE = {Two-{L}evel {R}estricted {A}dditive {S}chwarz {P}reconditioner
              {B}ased on {M}ultiscale {S}pectral {G}eneralized {FEM} for
              {H}eterogeneous {H}elmholtz {P}roblems},
   JOURNAL = {J. Sci. Comput.},
  FJOURNAL = {Journal of Scientific Computing},
    VOLUME = {105},
      YEAR = {2025},
    NUMBER = {3},
     PAGES = {Paper No. 99},
      ISSN = {0885-7474,1573-7691},
   MRCLASS = {65N30 (65F10 65N55)},
  MRNUMBER = {4992072},
       DOI = {10.1007/s10915-025-03138-y},
}

@article{Nataf2025,
      title={Coarse spaces for non-symmetric two-level preconditioners based on local extended generalized eigenproblems}, 
      author={Frédéric Nataf and Emile Parolin},
      journal={arXiv preprint arXiv:2404.02758},
      year={2025},
}

\end{document}